\documentclass{amsart}
\usepackage{amscd,euscript,color}
\usepackage[hyphens,spaces,obeyspaces]{url}
\usepackage{amsfonts}
\usepackage{theoremref}
\usepackage{amsmath}
\usepackage{amssymb}
\usepackage{amsthm}
\usepackage{makecell}
\usepackage{lipsum}
\usepackage{amsfonts}
\usepackage{bm}
\usepackage{tikz-cd}
\usepackage{dsfont}
\usepackage{mathtools}
\usepackage{tabularx}
\usepackage{hyperref}
\usepackage{graphicx}
\usepackage{enumerate}
\usepackage{comment}
\usepackage{color}
\hypersetup{
colorlinks=black,
linkcolor=black,
filecolor=black,      
urlcolor=black,
}
\newcommand{\sslash}{\mathbin{/\mkern-6mu/}}
\newcommand*{\bfrac}[2]{\genfrac{}{}{0pt}{}{#1}{#2}}
\newcommand{\Addresses}{{
  \bigskip
  \footnotesize

  Ruoxi Li, \textsc{Department of Mathematics, University of Pittsburgh, Pittsburgh, PA, USA}\par\nopagebreak
  \textit{E-mail address}: \texttt{rul44@pitt.edu}

  \medskip

  Rahul Singh, \textsc{Department of Mathematics, Louisiana State University, Baton Rouge, LA, USA}\par\nopagebreak
  \textit{E-mail address}: \texttt{rahulsingh@lsu.edu}
}}

\theoremstyle{plain}
\newtheorem{theorem}{Theorem}[section]
\newtheorem{proposition}[theorem]{Proposition}
\newtheorem{corollary}[theorem]{Corollary}
\newtheorem{lemma}[theorem]{Lemma}
\newtheorem{fact}[theorem]{Fact}
\newtheorem{properties}[theorem]{Properties}

\theoremstyle{definition}
\newtheorem{definition}[theorem]{Definition}

\theoremstyle{remark}
\newtheorem{remark}[theorem]{Remark}
\newtheorem{example}[theorem]{Example}

\title[Explicit formulas for mixed Hodge polynomials]{Explicit formulas for mixed Hodge polynomials of character varieties of nilpotent groups}
\author{Ruoxi Li, Rahul Singh}
\date{}
\DeclareMathOperator{\Hom}{Hom}

\DeclareMathOperator{\Spec}{Spec}

\DeclareMathOperator{\diag}{diag}

\makeatletter
\def\subsubsection{\@startsection{subsubsection}{3}%
\z@{.5\linespacing\@plus.2\linespacing}{-.5em}%
{\smallfont\bfseries}}
\makeatother

\begin{document}
\renewcommand\abstractname{\textbf{Abstract}}
\begin{abstract}
Let $\Hom^0(\Gamma,G)$ be the path-connected component of the identity representation of the variety of representations of a finitely generated nilpotent group $\Gamma$ into a connected reductive complex affine algebraic group $G$. With the formulas given by Florentino, Lawton and Silva, we provide explicit partition type formulas for the mixed Hodge polynomials of character varieties $\Hom^0(\Gamma,G)\sslash G$ when $G=Sp_{2n}$ and $G=SO_{n}$.
\end{abstract}
\maketitle
\tableofcontents
\section{Introduction}
Given a complex reductive group $G$, and a finitely presented group $\Gamma$, the $G$-character variety of $\Gamma$ is the (affine) geometric invariant theory (GIT) quotient
\[\mathcal{M}_\Gamma G=\Hom(\Gamma,G)\sslash G.\]
When the group $\Gamma$ is the fundamental group of a Riemann surface, the non-abelian Hodge correspondence (see, for example \cite{Sim}) shows that (components of) $\mathcal{M}_\Gamma G$ are homeomorphic to certain moduli spaces of $G$-Higgs bundles. They appear in connection to important problems in Mathematical Physics: for example, these spaces play an important role in the quantum field theory interpretation of the geometric Langlands correspondence, in the context of mirror symmetry (\cite{KW}). It is thus an important question to understand the geometry and topology of character varieties. In this article we look at the case of finitely generated nilpotent groups.


Let $\Gamma$ be a finitely generated nilpotent group of abelian rank $\geq 1$, that is, the free part of $\Gamma_{Ab}:=\Gamma/[\Gamma,\Gamma]$ is isomorphic to $\mathbb{Z}^r$ for some natural number $r\geq 1$. The character varieties of such nilpotent groups have been studied by Florentino-Lawton, Florentino-Lawton-Silva, Sikora, and Ramras-Stafa, among others (see, eg, \cite{FL,FLS,Sik,RS,Sta}). Let $\mathcal{M}_{\Gamma}^{0} G$ be the path-connected component of the identity representation in $\mathcal{M}_\Gamma G$. The explicit formulas for the mixed Hodge polynomials of $\mathcal{M}_{\Gamma}^{0} G$ is given by Florentino-Silva (\cite{FS}) and Florentino-Lawton-Silva (\cite{FLS}) in terms of determinant of the action of Weyl group of $G$ on the cohomolgy of maximal torus. \cite{FS} obtained even more explicit formulas in terms of partitions when $G=GL_{n}$ and $SL_{n}$. In this paper, we give similar partition type formulas when $G=Sp_{2n}$ and $SO_{n}$.

The outline of the article is as follows. Section 2 covers notations and preliminaries on character varieties and mixed
Hodge polynomials. In section 3 and 4, we explain
how to use the conjugacy classes of the wreath products to give the explicit formulas for the mixed Hodge polynomials of $\mathcal{M}_{\Gamma}^{0}G$ when $G=Sp_{2n},SO_{2n}$ and $SO_{2n+1}$. 
\subsection{Acknowledgements} We are thankful to Sean Lawton for bringing our attention to one of his recent paper with Carlos Florentino and Jaime Silva on character varieties of nilpotent groups that has corrected a mistake in the previous version of this article and helped us generalize to nilpotent groups. We would also like to thank Roman Fedorov for several comments and raising interesting questions.
\section{Preliminaries}
\subsection{Character varieties}
Given a finitely generated group $\Gamma$ and a complex reductive group $G$, the $G$-character variety of $\Gamma$ is defined to be the (affine) Geometric Invariant Theory (GIT) quotient (see \cite{MFK, Muk}; \cite{Sch} for topological aspects):
\[\mathcal{M}_\Gamma G=\Hom(\Gamma,G)\sslash G.\]
Note that $\Hom(\Gamma,G)$, the space of homomorphisms $\rho:\Gamma\to G$, is an affine variety,  as $\Gamma$ is defined by algebraic relations, and it is also a $G$-variety when considering the action of $G$ by conjugation on $\Hom(\Gamma,G)$. For example, when $G=GL_{n}$ and $\Gamma=\mathbb{Z}^r$, $\Hom(\Gamma,G)$ is isomorphic to the variety of $r$-tuples of mutually commuting invertible $n\times n$ matrices and the group $GL_{n}$ acts diagonally by conjugation on each matrix.

The GIT quotient above is the spectrum of the ring $\mathbb{C}[\Hom(\Gamma,G)]^G$ of $G$-invariant regular functions on $\Hom(\Gamma,G)$:
\[\Hom(\Gamma,G)\sslash G:=\Spec\left(\mathbb{C}\left[\Hom(\Gamma,G)\right]^G\right).\]
The GIT quotient does not parametrize all orbits, since some of them may not be distinguishable by invariant functions, see, for example \cite{Muk}. For detailed definitions and properties of general character varieties, we refer to \cite{FL,Sik}.

In general, $\mathcal{M}_\Gamma G$ is not path-connected in the analytic topology (hence not irreducible in the Zariski topology). Let us denote by
\[\mathcal{M}^0_\Gamma G:=\Hom^0(\Gamma,G)\sslash G\]
the path-connected component of the identity representation in $\mathcal{M}_\Gamma G$.

We have another interesting component when $\Gamma$ is a finitely presentable group whose abelianization is free, that is
\[\Gamma_{Ab}=\Gamma/[\Gamma,\Gamma]\cong\mathbb{Z}^r\]
for some $r\in\mathbb{N}$. We consider the following sequence with $T\subset G$ a fixed maximal torus
\[T^r\cong\Hom(\mathbb{Z}^r,T)\hookrightarrow\Hom(\Gamma_{Ab},G)\hookrightarrow\Hom(\Gamma,G)\twoheadrightarrow\mathcal{M}_\Gamma G.\]
Let us denote by $\mathcal{M}^T_\Gamma G\subset\mathcal{M}_\Gamma G$ the image of the composition above and call it the torus component. It follows that $\mathcal{M}^T_\Gamma G$ is an irreducible subvariety of $\mathcal{M}_\Gamma G$, being the image of the irreducible variety $T^r$ under a morphism.

Obviously, the identity representation ($\rho(\gamma)=e$ for all $\gamma\in\Gamma$, $e\in G$ being the identity) belongs to $\mathcal{M}^T_\Gamma G$. Since $\mathcal{M}^T_\Gamma G$ is path-connected (being irreducible over $\mathbb{C}$), we conclude that $\mathcal{M}^T_\Gamma G\subset\mathcal{M}^0_\Gamma G$ .

In this article, we will focus on the case when $\Gamma$ is a finitely generated nilpotent group of abelian rank $\geq 1$, that is, the free part of $\Gamma/[\Gamma,\Gamma]$ is isomorphic to $\mathbb{Z}^r$ for some natural number $r\geq 1$, the abelian rank of $\Gamma$. The topology and geometry of the corresponding $G$-character varieties
\[\mathcal{M}_{\Gamma}G=\Hom(\Gamma,G)\sslash G,\]
was studied in \cite{FL,FLS,Sik}, among others. 
\begin{remark}\label{remark}
    In the case when $\Gamma$ is a finitely generated nilpotent group of abelian rank $0$ (which is the case if and only if it is a finite group), one has $\mathcal{M}_\Gamma G\cong\mathcal{M}_\Gamma K$, where $K$ is a maximal compact subgroup of $G$. From \cite[Corollary 3.2]{BS}, we have that $\mathcal{M}_\Gamma G$ is a finite set and thus $\mathcal{M}_{\Gamma}^{0} G$ is a single point.
\end{remark}

In the case when $\Gamma=\mathbb{Z}^r$, $r\geq 1$, $\mathcal{M}^T_\Gamma G$ is an irreducible component of $\mathcal{M}_\Gamma G$ by \cite[Theorem 2.1]{Sik}. The following theorem in \cite[Theorem 3.4]{FLS} generalizes \cite[Remark 2.4]{Sik}, and completely answers a question raised in \cite[Problem 5.7]{FL}.

\begin{fact}[Sikora, Florentino-Lawton-Silva]
For every $r\in\mathbb{N}$ and a complex reductive group $G$
\[
\mathcal{M}^T_{\mathbb{Z}^r} G=\mathcal{M}^0_{\mathbb{Z}^r} G.
\]
\end{fact}
The following theorem (\cite[Theorem 1.2]{FL}) is important for us. 
\begin{fact}[Florentino-Lawton]\label{quotientoftorus}

If the semisimple part of $G$ is a product of $SL_n'$s and $Sp_n'$s, then $\mathcal{M}_{\mathbb{Z}^r}G$ is irreducible, so that $\mathcal{M}_{\mathbb{Z}^r}G=\mathcal{M}^0_{\mathbb{Z}^r}G=\mathcal{M}^T_{\mathbb{Z}^r} G.$
\end{fact}
\subsection{Mixed Hodge polynomials}
Mixed Hodge structures are generalizations to quasi-projective varieties of Hodge structures, which are used to study smooth projective varieties. Mixed Hodge polynomials encode numerical invariants coming from mixed Hodge structures in the form of polynomials, which can be thought of as deformations of the Poincar\'e polynomials and the so-called $E$-polynomials. Let us give more details.

For a smooth projective variety $X$, the cohomology of the underlying complex manifold satisfies the Hodge decomposition, that is,
\[
H^{k}(X,\mathbb{C})\cong\oplus_{p+q=k}H^{p,q}(X), \text{ and }\overline{H^{p,q}(X)}\cong H^{q,p}(X).
\]
Such a structure is called a pure Hodge structure of weight $k$ and can be reformulated in terms of filtrations.
Deligne (\cite{Del1,Del2}) discovered the existence of such structures for any complex quasi-projective variety $X$. More precisely, $H^{k}(X,\mathbb{C})$ has a natural decreasing filtration (the Hodge filtration) $F_{\bullet}$:
\[
H^{k}(X,\mathbb{C})=F_{0}\supset F_{1}\supset\ldots F_{k+1}=\{0\}
\]
and a natural increasing filtration (the weight filtration) $W^{\bullet}$:
\[
\{0\}=W^{-1}\subset W^{0}\subset \ldots W^{2k}=H^{k}(X,\mathbb{Q}),
\]
such that the induced filtration by $F_{\bullet}$ on the graded pieces $Gr^{W^{\bullet}_{\mathbb{C}}}=W_{\mathbb{C}}^{p+q}H^{k}(X,\mathbb{C})/W_{\mathbb{C}}^{p+q-1}H^{k}(X,\mathbb{C})$ is a pure Hodge structure of weight $p+q$ (here $W^{\bullet}_{\mathbb{C}}$ denotes the complexified weight filtration). Define 
\[
H^{k,p,q}(X):=Gr^{F_{\bullet}}_{p}Gr^{W_{\mathbb{C}}^{\bullet}}_{p+q}H^{k}(X,\mathbb{C}).
\]
Define the mixed Hodge numbers of $X$ as:
\[
h^{k,p,q}(X):=\dim_{\mathbb{C}}(H^{k,p,q}(X,\mathbb{C})).
\]
\begin{definition}
    Let $X$ be a quasi-projective variety over $\mathbb{C}$ of dimension $d$. The mixed Hodge polynomial of $X$ is defined to be
    \[
    \mu_{X}(t,u,v)=\sum_{k,p,q}h^{k,p,q}(X)t^{k}u^{p}v^{q}
    \]
\end{definition}
\begin{remark}
    \begin{enumerate}[(i)]
        \item Note that the total degree of $\mu_{X}$ is at most $2d$.
        \item If we specialize $u=v=1$, then we get the Poincar\'e polynomial and the specialization $t=-1$ gives us the $E$-polynomial of $X$. 
        \item In \cite[Theorem 5.2]{FS}, it is shown that the mixed Hodge structure on $\mathcal{M}_{\mathbb{Z}^r}G$ is round, that is, the only non-zero mixed Hodge numbers can be of the form $h^{k,k,k}(\mathcal{M}_{\mathbb{Z}^r}G)$, $k=0,\ldots,2\dim \mathcal{M}_{\mathbb{Z}^r}G$.
    \end{enumerate}
\end{remark}
Here are some common properties of mixed Hodge numbers:
\begin{properties}
    Let $X$ be a quasi-projective variety over $\mathbb{C}$. Then
    \begin{enumerate}[(i)]
        \item The mixed Hodge numbers are symmetric, that is, $h^{k,p,q}(X)=h^{k,q,p}(X)$. In particular, $\mu_{X}(t,u,v)=\mu_{X}(t,v,u)$.
        \item If $h^{k,q,p}(X)\neq 0$, then $p,q\leq k$. If $X$ is smooth, then $p+q\geq k$ and if $X$ is projective, then $p+q\leq k$.
    \end{enumerate}
\end{properties}
\begin{example}
    \begin{enumerate}[(i)]
        \item Let $X=\mathbb{C}^n$. We have $h^{0,0,0}=1$ and all other mixed Hodge numbers of $X$ are zero, therefore
        \[
        \mu_{\mathbb{C}^n}(t,u,v)=1.
        \]
        \item Let $X=\mathbb{C}\setminus
        \{0\}$. Since $\dim H^{1}(X,\mathbb{C})=1$, $X$ cannot have a pure Hodge structure since $h^{k,p,q}=h^{k,q,p}$. We have $h^{0,0,0}=h^{1,1,1}=1$ and all other mixed Hodge numbers of $X$ are zero. Therefore
        \[
        \mu_{\mathbb{C}\setminus
        \{0\}}(t,u,v)=1+tuv.\]
    \end{enumerate}
\end{example}
We have the following result about the mixed Hodge polynomial of $\mathcal{M}^{0}_{\Gamma}G$ (see \cite[Theorem 5.2]{FS}, \cite[Corollary 5.10]{FLS}):
\begin{fact}[Florentino-Lawton-Silva, Florentino-Silva]\label{generalform}
    Let $G$ be a complex reductive group and let $\Gamma$ be a finitely generated nilpotent group of abelian rank $r\geq 1$. Then $\mathcal{M}^{0}_{\Gamma}G$ is round and 
    \[
    \mu_{\mathcal{M}^{0}_{\Gamma}G}(t,u,v)
    =
    \mu_{\mathcal{M}^{0}_{\mathbb{Z}^r}G}(t,u,v)
    =
    \frac{1}{|W_{G}|}\sum_{w\in W_{G}}\big(\det(I+tuvA_{w})\big)^{r},
    \]
    where $A_{w}$ is the automorphism of $H^{1}(T,\mathbb{C})$ induced by the action of $w$ on $T$.
\end{fact}
A partition of a positive integer $n$ is a sequence of positive integers $(\lambda_1\ge\lambda_2\ge\cdots\ge\lambda_l>0)$ such that $\sum_{i=1}^l \lambda_i = n$. The $\lambda_i$'s are called the parts of the partition. For $n\in\mathbb{N}$, let $\mathcal{P}_{n}$ be the set of partitions of $n$. We denote $\lambda\in\mathcal{P}_n$ as
\[
\lambda=[1^{a_{1}}2^{a_{2}}\ldots n^{a_{n}}]
\]
where $a_{i}\geq 0$ denotes the number of parts of $\lambda$ equal to $i$.

There is an even more explicit formula (see \cite[Theorem 5.2]{FS}) of Fact \ref{generalform} when $G=GL_n$:
\begin{fact}[Florentino-Lawton-Silva, Florentino-Silva]
    Let $\Gamma$ be a finitely generated nilpotent group of abelian rank $r\geq 1$. The mixed Hodge polynomial of $\mathcal{M}_{\Gamma}^{0} GL_{n}$ is given by
    \[
    \mu_{\mathcal{M}_{\Gamma}^{0}GL_{n}}(t,u,v)
    =
    \sum_{\lambda\in\mathcal{P}_{n}}\prod_{i=1}^{n}\frac{\big(1-(-tuv)^{i}\big)^{a_{i}r}}{a_{i}!i^{a_{i}}},
    \]
    where $a_{i}$ denotes the number of parts of $\lambda$ equal to $i$.
\end{fact}
In the next two sections, we present similar partition type formulas for $\mu_{\mathcal{M}_{\Gamma}^{0}Sp_{2n}}(t,u,v)$ and $\mu_{\mathcal{M}^{0}_{\Gamma}SO_{n}}(t,u,v)$.
\begin{remark}
In the case when $\Gamma=\mathbb{Z}^r$, using Fact \ref{quotientoftorus} we get partition type formulas for $\mu_{\mathcal{M}_{\mathbb{Z}^r}GL_{n}}(t,u,v)$.  
\end{remark}
\section{Symplectic and odd orthogonal groups}
\subsection{Symplectic groups}

Let $(V,\omega)$ be a symplectic vector space over $\mathbb{C}$ of dimension $2n$, that is, $\omega$ is a non-degenerate alternating bilinear form on $V$. The group $Sp_{2n}$ is defined to be the group of $A\in SL(V)$ such that $A$ preserves $\omega$. It is known that over $\mathbb{C}$ all the symplectic vector spaces are isomorphic and therefore we fix our $V$ to be $\mathbb{C}^{2n}$ with $\omega(e_{i},e_{i+n})=-\omega(e_{i+n},e_{i})=1$ if $1\leq i\leq n$ and $\omega(e_{i},e_{j})=0$ otherwise, and identify $Sp_{2n}$ with $A\in SL_{2n}(\mathbb{C})$ such that $A J A^t=J$, where
\[
J=\begin{bmatrix}
0
&
I_{n}
\\
-I_{n}
&
0
\end{bmatrix}.
\]
We will work with this matrix version of symplectic groups from now.

We choose our maximal torus $T_{Sp_{2n}}$ of $Sp_{2n}$ to be the standard one:
\[
T_{Sp_{2n}}=\{\diag(a_{1},\ldots,a_{n},a_{1}^{-1},\ldots,a_{n}^{-1}):a_{i}\in\mathbb{C}^{\times}, 1\leq i\leq n\}.
\]
The Weyl group $W_{Sp_{2n}}\cong N_{Sp_{2n}}(T_{Sp_{2n}})/T_{Sp_{2n}}$ acts on $T_{Sp_{2n}}$ by permuting $a_{1},\ldots,a_{n}$ (and $a_{1}^{-1},\ldots,a_{n}^{-1}$ accordingly, that is, $\sigma\cdot a_{i}^{-1}=a_{\sigma(i)}^{-1}$) and inverting some of the entries of $(a_{1},\ldots,a_{n})$ and the corresponding entries of $(a_{1}^{-1},\ldots,a_{n}^{-1})$. Thus, $W_{Sp_{2n}}$ can be identified with the wreath product $C_2\wr S_n$ of $C_2$ with $S_n$. Let us recall the definition of the wreath product $C_2\wr S_n$.

\begin{definition}
For a positive integer $n$, let $C_2^n$ denote the direct product of $n$-copies of cyclic groups of order $2$. Consider the following action of $S_n$ on $C_{2}^{n}$
\[
\sigma\cdot(a_1,\ldots,a_n)=(a_{\sigma^{-1}(1)},\ldots,a_{\sigma^{-1}(n)}), \quad 
(a_1,\ldots,a_n)\in C_{2}^{n}, \sigma\in S_n
.\]
Then the semidirect product of $C_2^n$ with $S_n$, with respect to the above action, is called the wreath product of $C_2$ and $S_n$.
\end{definition}
We write an element of $C_2\wr S_n$ as $[a_1,\ldots,a_n;\sigma]$, where $a_1,\ldots,a_n\in C_2$ and $\sigma\in S_n$. Here we take $C_2=\{-1,1\}$. 

Let $\mathbf{a}=[a_1,\ldots,a_n;\sigma],\mathbf{b}=[b_1,\ldots,b_n;\pi]\in C_2\wr S_n$, then $\mathbf{a}\mathbf{b}$ is the following element
\[[a_1b_{\sigma^{-1}(1)},\ldots,a_nb_{\sigma^{-1}(n)};\sigma\cdot\pi].\]
The inverse of $\mathbf{a}$ is $\mathbf{a}^{-1}=[a^{-1}_{\sigma(1)},\ldots,a^{-1}_{\sigma(n)};\sigma^{-1}]$.
\begin{definition}
    A signed partition $\lambda^\pm$ of a positive integer $n$ is a pair $(\lambda^{+},\lambda^{-})$, where $\lambda^{+}\in\mathcal{P}_{k}$, $\lambda^{-}\in\mathcal{P}_{l}$ for some $k,l\in\mathbb{N}_{\geq0}$ such that $k+l=n$. We will denote by $\mathcal{P}^{\pm}_n$ the set of signed partitions of $n$. For $\lambda^{\pm}=(\lambda^{+},\lambda^{-})\in\mathcal{P}^{\pm}_n$, we write
$$    \lambda^{\pm}=[1^{a_{1}}\overline{1}^{b_{1}}2^{a_{2}}\overline{2}^{b_{2}}\ldots m^{a_{m}}\overline{m}^{b_{m}}],
$$   
 where $a_{i}$ (resp. $b_{i}$) is the number of parts of $\lambda^+$ (resp. $\lambda^-$) equal to $i$. Of course, we must have $\sum_{i}i(a_{i}+b_{i})=n$.
The parts of $\lambda^+$ (resp. $\lambda^-$) are called parts with positive (resp. negative) sign.

 Sometimes we will also write $\lambda^{\pm}$ in the following form
 \[
 \lambda^{\pm}=[k_1^{s_1}\overline{k}_1^{t_1}\cdots k_u^{s_u}\overline{k}_u^{t_u}m_1^{x_1}\overline{m}_1^{y_1}\cdots m_v^{x_v}\overline{m}_v^{y_v}],
 \]
 where $k_1<\cdots<k_u$ are odd numbers, and $m_1<\cdots<m_v$ are even numbers with $s_{i}$ (resp. $t_{i}$) equal to the number of parts with positive sign (resp. negative sign) of $\lambda^{\pm}$ equal to $k_{i}$ ($x_{i}$'s and $y_{i}$'s are defined similarly).
\end{definition}
 \begin{example}
     There are five signed partitions of $2$ :
\[
[1^2],
[1\overline{1}],
[\overline{1}^2],
[2],
[\overline{2}].
\]
\end{example}
The conjugacy classes of $C_2\wr S_n$ have been determined by Carter in \cite{Car}. The conjugacy classes of $C_2\wr S_n$ are in one to one correspondence with the signed partitions of $n$. To state this correspondence, we recall the notion of cycle products.
\begin{definition}
Let $\lambda=(\lambda_1\geq\lambda_2\geq\ldots\geq\lambda_l>0)$ be a partition of $n$. Let $\sigma\in S_n$ be a
permutation of cycle type $\lambda$, say $\sigma= (i_1,\ldots,i_{\lambda_1})\cdots(k_1,\ldots,k_{\lambda_l})$, and let $(a_1,\ldots,a_n)\in C_{2}^{n}$. Consider $\mathbf{a}= [a_1,\ldots,a_n;\sigma]\in C_2\wr S_n$. The cycle products of $\mathbf{a}$ are defined to be the products $\prod_{j=1}^{\lambda_1}a_{i_j},\ldots,\prod_{j=1}^{\lambda_l}a_{k_j}$.   
\end{definition}
For any element $\mathbf{a}\in C_2\wr S_n$, the cycle products of $\mathbf{a}$ are either $-1$ or $1$. We have the following proposition in \cite[Section 2]{MS}.
\begin{fact}[Mishra-Srinivasan]\label{MS}
Let $\mathbf{a}=[a_1,\ldots,a_n;\sigma], \mathbf{b}=[b_1,\ldots,b_n;\pi]\in C_2\wr S_n$. Then $\mathbf{a}$ is conjugate to $\mathbf{b}$ in $C_2\wr S_n$ if and only if $\sigma$ is conjugate to $\pi$, and if $\tau\in S_n$ is such that $\tau\sigma\tau^{-1}=\pi$, then for each cycle $(i_1,\ldots,i_{\lambda_j})$ in $\sigma$, where $\lambda=(\lambda_1\geq\lambda_2\geq\ldots\geq\lambda_l>0$) is the cycle type of $\sigma$ (and hence also of $\pi$), the cycle product $a_{i_1}\cdots a_{i_{\lambda_j}}$ of $\mathbf{a}$ is equal to the cycle product $b_{\tau(i_1)}\cdots a_{\tau(i_{\lambda_j})}$ of $\mathbf{b}$, $1\leq j\leq l$.
\end{fact}
Let $\lambda=(\lambda_1\geq\lambda_2\geq\ldots\geq\lambda_l>0$) be a partition of $n$. Let $\mathbf{a}=[a_1,\ldots,a_n;\sigma]$ be an element of $C_2\wr S_n$ with $\sigma$ being of cycle type $\lambda$. For each $j=1,\ldots,l$, we leave $\lambda_j$ as it is, if the associated cycle product of $\mathbf{a}$ is $1$. We write $\overline{\lambda}_j$, if the associated cycle product of $\mathbf{a}$ is $-1$. This gives a way of associating a signed partition of $n$ with the elements of the group $C_2\wr S_n$. 

Now Fact \ref{MS} can be reformulated as:
\begin{proposition}\label{labellingconjugacyclasses}
Let $\mathbf{a}, \mathbf{b}\in C_2\wr S_n$. Then $\mathbf{a}$  and $\mathbf{b}$ are conjugate in $C_2\wr S_n$ if and only if the associated signed partitions are equal.
\end{proposition}
\begin{example}
Consider $C_2\wr S_6$. Let $\mathbf{x}=[1,-1,-1,1,1,-1;(145)(26)]\in C_2\wr S_6$. The $\sigma$ here is $(145)(26)$, which is of cycle type $\lambda=[1^12^13^1]$, the cycle products are $x_1x_4x_5=1\cdot 1\cdot 1=1,x_2x_6=(-1)(-1)=1$ and $x_3=-1$. Thus the signed partition associated to $\mathbf{x}$ is $[\overline{1}^12^13^1]$.
\end{example}
In the following example we shall list the conjugacy classes for $C_2\wr S_2$, with their
representatives.
\begin{example}
$(C_2\wr S_2)$. The conjugacy classes of $C_2\wr S_2$ along with a representative from each are listed below

\centering\begin{tabularx}{0.6\textwidth} { 
  | >{\centering\arraybackslash}X 
  | >{\centering\arraybackslash}X 
  | 
    }
 \hline
 Class & Representative\\
 \hline
 $[1^2]$ & $[1,1;e]$ \\
 \hline
 $[1^1\overline{1}^1]$ & $[-1,1;e]$ \\
 \hline
 $[\overline{1}^2]$ & $[-1,-1;e]$ \\
 \hline
 $[2^1]$ & $[1,1;(12)]$ \\
 \hline
 $[\overline{2}^1]$ & $[1,-1;(12)]$ \\
 \hline
\end{tabularx}
\end{example}

\begin{example}
$(C_2\wr S_3)$. The conjugacy classes of $C_2\wr S_3$ are listed below

\centering\begin{tabularx}{0.6\textwidth} { 
  | >{\centering\arraybackslash}X 
  | >{\centering\arraybackslash}X
  | 
    }
 \hline
 Class & Representative \\
 \hline
 $1^3$ & $[1,1,1;e]$ =\\
 \hline
 $\overline{1}^11^2$ & $[-1,1,1;e]$ \\
 \hline
 $\overline{1}^21^1$ & $[1,-1,-1;e]$ \\
 \hline
 $\overline{1}^3$ & $[-1,-1,-1;e]$ \\
 \hline
 $2^11^1$ & $[1,1,1;(12)]$\\
 \hline
 $2^1\overline{1}^1$ & $[1,1,-1;(12)]$\\
 \hline
 $\overline{2}^11^1$ & $[1,-1,1;(12)]$ \\
 \hline
 $\overline{2}^1\overline{1}^1$ & $[1,-1,-1;(12)]$ \\
 \hline
 $3^1$ & $[1,1,1;(123)]$ \\
 \hline
 $\overline{3}^1$ & $[-1,-1,-1;(123)]$\\
 \hline
\end{tabularx}
\end{example}
Now for each conjugacy class of $W_{Sp_{2n}}$, we select some representatives. Before describing them, we shall set the following notations:
\begin{itemize}
    \item $\mathbf{1}_k$ denotes the sequence $(1,1,\ldots,1)$ of $k$ $1$'s.
    \item $-\mathbf{1}_k$ denotes the sequence $(-1,-1,\ldots,-1)$ of $k$ $-1$'s.
    \item $\overline{\mathbf{1}}_k$ denotes the sequence $(-1,\underbrace{1,\ldots,1}_{k-1\text{-copies}})$.
    \item For $k,l,n\in\mathbb{N}$ such that $kl=n$, let $\sigma_{k,l}\in S_n$ denote the permutation $(1,2,\ldots,k)(k+1,k+2,\ldots,2k)\cdots(n-k+1,\ldots n)\in S_n$. In particular, as in the following table, $\sigma_{n,1}$ denotes the $n$-cycle $(1,2,\ldots,n)\in S_n$. When $l=s+t$, $\sigma_{k,s}$ denotes the permutation in $S_{k,s}$ and $\sigma_{k,t}$ denotes the permutation in $S_{k,t}$.
    \item $\pi_k$ denotes the $k$-cycle $(1,2,\ldots,k)\in S_n$.
\end{itemize}
We summarize the centeralizers of some of the basic conjugacy classes computed in \cite[Proposition 3.1 to 3.9]{KS} in the table below. For an abstract group $H$ and an element $h\in H$, we denote by $Z_H(h)$ the centralizer of $h$ in $H$.
\\
\\
\begin{tabularx}{1\textwidth} { 
   >{\hsize=5\hsize\linewidth=\hsize\centering\arraybackslash}X 
  | >
  {\hsize=10\hsize\linewidth=\hsize\centering\arraybackslash}X 
  | >{\hsize=23\hsize\linewidth=\hsize\centering\arraybackslash}X
    }
 \hline
  Conjugacy Class  & Representative & 
  Centralizer  \\
 \hline
 $n^1$  & $[\mathbf{1}_n;\sigma_{n,1}]$ & $\{[\mathbf{1}_n;\sigma^k_{n,1}],[-\mathbf{1}_n;\sigma^k_{n,1}]\ \vert\  k=0,1,\ldots,n-1\}$    \\
 & & \\
 $\overline{n}^1$  &$[-\mathbf{1}_n;\sigma_{n,1}]$ & $\{[\mathbf{1}_n;\sigma^k_{n,1}],[-\mathbf{1}_n;\sigma^k_{n,1}]\ \vert\  k=0,1,\ldots,n-1\}$    \\
 ($n$ odd) & & \\
 $\overline{n}^1$   & $[\overline{\mathbf{1}}_n;\sigma_{n,1}]$ & Cyclic group $\langle[\overline{\mathbf{1}}_n;\sigma_{n,1}]\rangle$ of order $2n$    \\
($n$ even) &  & \\
 $k^l\ (l\ge 2)$ &$[\underbrace{\mathbf{1}_k,\ldots,\mathbf{1}_k}_{l\text{-copies}};\sigma_{k,l}]$ &  $\left\{ [\underbrace{h_1,\ldots,h_1}_{k\text{-copies}},\underbrace{h_2,\ldots,h_2}_{k\text{-copies}},\ldots,\underbrace{h_l,\ldots,h_l}_{k\text{-copies}};\tau]\ \vert\ h_i=\pm 1,\tau\in Z_{S_n}(\sigma_{k,l})\right\}$   \\
  & & \\
 $\makecell{\overline{k}^l (l\ge 2,\\  k\ \text{odd})}$  &$[-\underbrace{\mathbf{1}_k,\ldots,-\mathbf{1}_k}_{l\text{-copies}};\sigma_{k,l}]$ & $\left\{ [\underbrace{h_1,\ldots,h_1}_{k\text{-copies}},\underbrace{h_2,\ldots,h_2}_{k\text{-copies}},\ldots,\underbrace{h_l,\ldots,h_l}_{k\text{-copies}};\tau]\ \vert\ h_i=\pm 1,\tau\in Z_{S_n}(\sigma_{k,l})\right\}$    \\
 & & \\
 $\makecell{\overline{k}^l (l\ge 2,\\  k\ \text{even})}$ &$[\underbrace{\overline{\mathbf{1}}_k,\ldots,\overline{\mathbf{1}}_k}_{l\text{-copies}};\sigma_{k,l}]$ & $\left\{ \bfrac{[\underbrace{-h_1,\ldots,-h_1}_{m_1\text{-copies}},\underbrace{h_1,\ldots,h_1}_{k-m_1\text{-copies}},\ldots,\underbrace{-h_l,\ldots,-h_l}_{m_l\text{-copies}},\underbrace{h_l,\ldots,h_l}_{k-m_l\text{-copies}};\tau]\ \vert }{h_i=\pm 1,\tau=[\pi^{m_1}_k,\ldots,\pi^{m_l}_k;\rho]\in Z_{S_n}(\sigma_{k,l})}\right\}$    \\
 &  & \\
 \makecell{$k^s\overline{k}^t$\\ ($k$ odd)} &$\left[\bfrac{\underbrace{\mathbf{1}_k,\ldots,\mathbf{1}_k,}_{s\text{-copies}}}{-\underbrace{\mathbf{1}_k,\ldots,-\mathbf{1}_k}_{t\text{-copies}}};\sigma_{k,s+t}\right]$ & $\left\{ \bfrac{[\underbrace{h_1,\ldots,h_1}_{k\text{-copies}},\underbrace{h_2,\ldots,h_2}_{k\text{-copies}},\ldots,\underbrace{h_{s+t},\ldots,h_{s+t}}_{k\text{-copies}};\tau_s\tau_t]\ \vert}{h_i=\pm 1,\tau_s\in Z_{S_{sk}}(\sigma_{k,s}),\tau_t\in Z_{S_{tk}}(\sigma_{k,t})}\right\}$    \\
  & & \\
 \makecell{$k^s\overline{k}^t$\\ ($k$ even)}& $\left[\bfrac{\underbrace{\mathbf{1}_k,\ldots,\mathbf{1}_k,}_{s\text{-copies}}}{\underbrace{\overline{\mathbf{1}}_{k},\ldots,\overline{\mathbf{1}}_k}_{t\text{-copies}}};\sigma_{k,s+t}\right]$ & $\left\{ \left[\bfrac{\underbrace{h_1,\ldots,h_1}_{k-\text{-copies}},\ldots,\underbrace{h_s,\ldots,h_s}_{k-\text{-copies}},\underbrace{-h_{s+1},\ldots,-h_{s+1}}_{m_1\text{-copies}},}{\underbrace{h_{s+1},\ldots,h_{s+1}}_{k-m_1\text{-copies}},\ldots,\underbrace{-h_l,\ldots,-h_l}_{m_t\text{-copies}},\underbrace{h_l,\ldots,h_l}_{k-m_t\text{-copies}}};\tau_s\tau_t \right]\vert \bfrac{\tau_s\in Z_{S_{sk}}(\sigma_{k,s})}{\tau_t\in Z_{S_{tk}}(\sigma_{k,t})}\right\}$    \\
  & & \\
\hline
\end{tabularx}

More generally, we have the following result (\cite[Theorem 3.10]{KS}).
\begin{fact}[Kaur-Sharma]\label{centralizer}
Let $[k_1^{k_1}\cdots k_u^{l_u}m_1^{r_1}\cdots m_v^{r_v}]$ be a partition of $n$, where $k_1<\cdots<k_u$ are odd numbers, and $m_1<\cdots<m_v$ are even numbers. For $1\le i\le u$, let $s_i,t_i$ be nonnegative integers such that $s_i+t_i=l_i$, and likewise for $1\le j\le v$, we take $x_j+y_j=r_j$. Now, consider the conjugacy class of $C_2\wr S_n$ given by the following signed partition:
\[
[
k_1^{s_1}\overline{k}_1^{t_1}\cdots k_u^{s_u}\overline{k}_u^{t_u}m_1^{x_1}\overline{m}_1^{y_1}\cdots m_v^{x_v}\overline{m}_v^{y_v}
]
.
\]
Let us choose the following representative from this conjugacy class:
\[
\mathbf{a}=[\underbrace{\mathbf{1}_{k_{1}},\ldots,\mathbf{1}_{k_{1}}}_{s_{1}\text{ times }},\underbrace{\tilde{\mathbf{1}}_{k_{1}},\ldots,\tilde{\mathbf{1}}_{k_{1}}}_{t_{1}\text{ times}},\ldots,\underbrace{\mathbf{1}_{m_{v}},\ldots\mathbf{1}_{m_{v}}}_{x_{v}\text{ times}},\underbrace{\tilde{\mathbf{1}}_{m_{v}},\ldots,\tilde{\mathbf{1}}_{m_{v}}}_{y_{v}\text{ times}};\sigma]
\]
where $\tilde{\mathbf{1}}_k=-\mathbf{1}_k$ (resp. $\overline{\mathbf{1}}_k$) if $k$ is odd (resp. even) and $\sigma\in S_{n}$ is defined as
\[
(1,\ldots,k_{1})\ldots(k_{1}s_{1}+k_{1}(t_{1}-1)+1,\ldots,k_{1}s_{1}+k_{1}t_{1})\ldots(n-m_{v}+1,\ldots,n).
\]
Then the centralizer $Z_{C_2\wr S_n}(\mathbf{a})$ is isomorphic to
\[\left(\prod_{i=1}^u\left(Z_{C_2\wr S_{s_ik_i}}(k_i^{s_i})\times Z_{C_2\wr S_{t_ik_i}}(\overline{k}_i^{t_i})\right)\right)\times\left(\prod_{j=1}^v\left(Z_{C_2\wr S_{x_jm_j}}(m_j^{x_j})\times Z_{C_2\wr S_{y_jm_j}}(\overline{m}_j^{y_j})\right)\right).\]
Moreover, this centralizer is isomorphic to the following group
\[\left(\prod_{i=1}^u\left((C_2\times C_{k_i})\wr S_{s_i}\right)\times \left((C_2\times C_{k_i})\wr S_{t_i}\right)\right)\times\left(\prod_{j=1}^v\left((C_2\times C_{m_j})\wr S_{x_j}\right)\times \left((C_2\times C_{m_j})\wr S_{y_j}\right)\right).\]
\end{fact}
For $n\geq 2$, let $A_{n}$ and $B_{n}$ denote the following $n\times n$ matrices:
\[
A_{n}=
\begin{bmatrix}
    0&0&\ldots&0&1
    \\
    1&0&\ldots&0&0
    \\
    0&1&\ldots&0&0
    \\
    \vdots&&\ddots&0&0
    \\
    0&\ldots&0&1&0
\end{bmatrix}
,\quad
B_{n}=
\begin{bmatrix}
    0&0&\ldots&0&-1
    \\
    1&0&\ldots&0&0
    \\
    0&1&\ldots&0&0
    \\
    \vdots&&\ddots&0&0
    \\
    0&\ldots&0&1&0
\end{bmatrix}
\]
We define $A_{1}=[1]$ and $B_{1}=[-1]$.

We need the following lemma, which will be used in obtaining the desired formulas.
\begin{lemma}\label{determinant}
 We have $\det(I-xA_{n})=1-x^{n}$ and $\det(I-xB_{n})=x^{n}+1$.   
\end{lemma}
\begin{proof}
    The proof when $n=1$ is straightforward, so let us assume that $n\geq 2$. We note that $A_{n}$ is the companion matrix of the polynomial $p(\lambda)=\lambda^{n}-1$, therefore $\det(\lambda I-A_{n})=\lambda^{n}-1$. Now the desired formula for $A_{n}$ follows upon substituting $\lambda=x^{-1}$ in $\lambda^{n}\det(I-\lambda^{-1}A)=\lambda^{n}-1$. The proof of $\det(I-xB_{n})=x^{n}+1$ is similar. 
\end{proof}
The statement (without a proof) of the following theorem is given in \cite[Theorem 3]{FNSZ}.
\begin{theorem}\label{sp2n}
Let $\Gamma$ be a finitely generated nilpotent group of abelian rank $r\geq 1$. The mixed Hodge polynomial of $\mathcal{M}_{\mathbb{Z}^r}Sp_{2n}$ has the following explicit form:
    \[
    \mu_{\mathcal{M}_{\Gamma}^{0}Sp_{2n}}(t,u,v)=\mu_{\mathcal{M}_{\mathbb{Z}^r}Sp_{2n}}(t,u,v)
    =
    \sum_{\lambda^{\pm}\in\mathcal{P}^{\pm}_n}\frac{\prod_{i=1}^{n}(1-(-tuv)^{i})^{a_{i}r}\prod_{j=1}^{n}(1+(-tuv)^{j})^{b_{j}r}}{\prod_{i=1}^{n}(2i)^{a_{i}}a_{i}!\prod_{j=1}^{n}(2j)^{b_{j}}b_{j}!},
    \]
    where $a_{i}$ (resp. $b_{i}$) is the number of parts with positive sign (resp. negative sign) of $\lambda^{\pm}$ equal to $i$.
\end{theorem}
\begin{proof}
Since $\det(I+tuvA_{w})=\det(I+tuvA_{w'})$ if $w$ and $w'$ are conjugate, using Proposition \ref{labellingconjugacyclasses} and Fact \ref{quotientoftorus}, the formula in Fact \ref{generalform} for $G=Sp_{2n}$ can be written as
\[
    \mu_{\mathcal{M}_{\Gamma}^{0}Sp_{2n}}(t,u,v)
    =
    \mu_{\mathcal{M}_{\mathbb{Z}^r}Sp_{2n}}(t,u,v)
    =
    \sum_{\lambda^{\pm}\in\mathcal{P}^{\pm}_n}\frac{\big(\det(I+tuvA_{w_{\lambda^{\pm}}})\big)^{r}}{|Z_{W_{Sp_{2n}}}(w_{\lambda^{\pm}})|},
    \]
where for each $\lambda^{\pm}=[1^{a_{1}}\overline{1}^{b_{1}}2^{a_{2}}\overline{2}^{b_{2}}\ldots n^{a_{n}}\overline{n}^{b_{n}}]\in\mathcal{P}^{\pm}_n$ we choose a representative $w_{\lambda^{\pm}}\in W_{Sp_{2n}}$ of the conjugacy class corresponding to $\lambda^{\pm}$ defined as follows:
\[
w_{\lambda^{\pm}}=[\underbrace{\mathbf{1}_1,\ldots,\mathbf{1}_1}_{a_{1}\text{ times }},\underbrace{\tilde{\mathbf{1}}_1,\ldots,\tilde{\mathbf{1}}_1}_{b_{1}\text{ times}},\ldots,\underbrace{\mathbf{1}_n,\ldots\mathbf{1}_n}_{a_{n}\text{ times}},\underbrace{\tilde{\mathbf{1}}_n,\ldots,\tilde{\mathbf{1}}_n}_{b_{n}\text{ times}};\sigma_{\lambda^{\pm}}]
\]
where $\tilde{\mathbf{1}}_k=-\mathbf{1}_k$ (resp. $\overline{\mathbf{1}}_k$) if $k$ is odd (resp. even) and $\sigma_{\lambda^{\pm}}\in S_{n}$ is defined as
\[
(1)(2)\ldots(a_{1}+b_{1})(a_{1}+b_{1}+1,a_{1}+b_{1}+2)\ldots(a_{1}+b_{1}+2a_{2}+2b_{2}-1,a_{1}+b_{1}+2a_{2}+2b_{2}),\ldots,(n-m+1,\ldots,n),
\]
that is, $\sigma_{\lambda^{\pm}}\in S_{n}$ is a permutation with partition type $[1^{a_{1}+b_{1}}\ldots n^{a_{n}+b_{n}}]$.

The automorphism of $H^{1}(T_{Sp_{2n}},\mathbb{C})$ induced by the action of $w_{\lambda^{\pm}}$ on $T_{Sp_{2n}}$ can be written in the following block matrix form
\[
\begin{bmatrix}
    A_{1}&0&\cdots&0&0&0&0&0&0&0
    \\
    0&\ddots&0&0&0&0&0&0&0&0
    \\
    0&\cdots&A_{1}&0&0&0&0&0&0&0
    \\
    0&\cdots&0&\tilde{B_{1}}&0&0&0&0&0&0
    \\
    0&\cdots&0&0&\ddots&0&0&0&0&0
    \\
    0&\cdots&0&0&0&\tilde{B_{1}}&0&0&0&0
    \\
    0&\cdots&0&0&0&0&\ddots&0&0&0
    \\
    0&\cdots&0&0&0&0&\cdots&A_{n}&0&0
    \\
    0&\cdots&0&0&0&0&0&0&\ddots&0
    \\
    0&\cdots&0&0&0&0&0&0&\cdots&\tilde{B_{n}},
\end{bmatrix}
\]
where $\tilde{B}_{k}=-A_{k}$ (resp. $B_{k}$) if $k$ is odd (resp. even).
Thus by Lemma \ref{determinant}, we have
\[
\det(I+tuvA_{w_{\lambda^{\pm}}})=\prod_{i=1}^{n}(1-(-tuv)^{i})^{a_{i}}\prod_{j=1}^{n}(1+(-tuv)^{j})^{b_{j}}.
\]
By Fact \ref{centralizer}, we have $|Z_{W_{Sp_{2n}}}(w_{\lambda^{\pm}})|=\prod_{i=1}^{n}(2i)^{a_{i}}a_{i}!\prod_{j=1}^{n}(2j)^{b_{j}}b_{j}!$ and the result follows.
\end{proof}
\subsection{Odd orthogonal groups}
Let $(V,\psi)$ be a vector space over $\mathbb{C}$ of dimension $2n+1$ equipped with a non-degenerate symmetric bilinear form $\psi$ on $V$. The group $SO_{2n+1}$ is defined to be the group of $A\in SL(V)$ such that $A$ preserves $\psi$. It is known that over $\mathbb{C}$ all such spaces $(V,\psi)$ are isomorphic and therefore we fix our $V$ to be $\mathbb{C}^{2n+1}$ with  $\psi(e_{i},e_{i+n+1})=\psi(e_{i+n+1},e_{i})=1$, $1\leq i\leq n$, $\psi(e_{n+1},e_{n+1})=1$ and $0$ otherwise; we identify $SO_{2n+1}$ with $A\in SL_{2n+1}(\mathbb{C})$ such that $A K A^t=K$, where
\[
K=\begin{bmatrix}
0
&
0
&
I_{n}
\\
0
&
1
&
0
\\
I_{n}
&
0
&
0
\end{bmatrix}.
\]
We will work with this matrix version of odd orthogonal groups from now.

We choose our maximal torus $T_{SO_{2n+1}}$ of $SO_{2n+1}$ to be the standard one:
\[
T_{SO_{2n+1}}=\{\diag(a_{1},\ldots,a_{n},1,a_{1}^{-1},\ldots,a_{n}^{-1}):a_{i}\in\mathbb{C}^{\times}, 1\leq i\leq n\}.
\]
The Weyl group $W_{SO_{2n+1}}\cong N_{SO_{2n+1}}(T_{SO_{2n+1}})/T_{SO_{2n+1}}$ acts on $T_{SO_{2n+1}}$ by fixing 
$(n+1)\times(n+1)$-st entry, permuting $a_{1},\ldots,a_{n}$ (and $a_{1}^{-1},\ldots,a_{n}^{-1}$ accordingly, that is, $\sigma\cdot a_{i}^{-1}=a_{\sigma(i)}^{-1}$) and inverting some of the entries of $(a_{1},\ldots,a_{n})$ and the corresponding entries of $(a_{1}^{-1},\ldots,a_{n}^{-1})$. Thus, $W_{SO_{2n}}$ can be identified with the wreath product $C_2\wr S_n$. 

The natural morphism
\[
\phi: T_{Sp_{2n}}\rightarrow T_{SO_{2n+1}}
\]
\[
\diag(a_{1},\ldots,a_{n},a_{1}^{-1},\ldots,a_{n}^{-1})\mapsto \diag(a_{1},\ldots,a_{n},1,a_{1}^{-1},\ldots,a_{n}^{-1})
\]
is an isomorphism and satisfies $\phi(\mathbf{a}\cdot A)=\mathbf{a}\cdot\phi(A)$, $\mathbf{a}\in C_{2}\wr S_n$, $A\in T_{Sp_{2n}}$.
By Fact \ref{generalform} and Theorem \ref{sp2n}, we have
\begin{corollary}
Let $\Gamma$ be a finitely generated nilpotent group of abelian rank $r\geq 1$. The mixed Hodge polynomial of $\mathcal{M}^{0}_{\Gamma}SO_{2n+1}$ has the following explicit form:
    \[
    \mu_{\mathcal{M}_{\Gamma}^{0}SO_{2n+1}}(t,u,v)
    =
    \mu_{\mathcal{M}^{0}_{\mathbb{Z}^r}SO_{2n+1}}(t,u,v)=\sum_{\lambda^{\pm}\in\mathcal{P}^{\pm}_n}\frac{\prod_{i=1}^{n}(1-(-tuv)^{i})^{a_{i}r}\prod_{j=1}^{n}(1+(-tuv)^{i})^{b_{j}r}}{\prod_{i=1}^{n}(2i)^{a_{i}}a_{i}!\prod_{j=1}^{n}(2j)^{b_{j}}b_{j}!}
    .
    \]
\end{corollary}
\section{Even orthogonal groups}
Let $(V,\psi)$ be a vector space over $\mathbb{C}$ of dimension $2n$ equipped with a non-degenerate symmetric bilinear form $\psi$ on $V$. The group $SO_{2n}$ is defined to be the group of $A\in SL(V)$ such that $A$ preserves $\psi$. It is known that over $\mathbb{C}$ all such spaces $(V,\psi)$ are isomorphic and therefore we fix our $V$ to be $\mathbb{C}^{2n}$ with $\psi(e_{i},e_{i+n})=\psi(e_{i+n},e_{i})=1$ if $1\leq i\leq n$ and $0$ otherwise; we identify $SO_{2n}$ with $A\in SL_{2n}(\mathbb{C})$ such that $A L A^t=L$, where
\[
L=\begin{bmatrix}
0
&
I_{n}
\\
I_{n}
&
0
\end{bmatrix}.
\]
We will work with this matrix version of even orthogonal groups from now.

We choose our maximal torus $T$ of $SO_{2n}$ to be the standard one:
\[
T=\{\diag(a_{1},\ldots,a_{n},a_{1}^{-1},\ldots,a_{n}^{-1}):a_{i}\in\mathbb{C}^{\times}, 1\leq i\leq n\}.
\]

The Weyl group $W_{SO_{2n}}\cong N_{SO_{2n}}(T)/T$ acts on $T$ by  permuting $a_{1},\ldots,a_{n}$ (and $a_{1}^{-1},\ldots,a_{n}^{-1}$ accordingly, that is, $\sigma\cdot a_{i}^{-1}=a_{\sigma(i)}^{-1}$) and inverting an even number of entries of $(a_{1},\ldots,a_{n})$ and the corresponding entries of $(a_{1}^{-1},\ldots,a_{n}^{-1})$. Thus, $W_{SO_{2n}}$ can be identified with a subgroup of $C_2\wr S_n$. In fact we can describe this subgroup as follows, $\mathbf{a}= [a_1,\ldots,a_n;\sigma]\in C_2\wr S_n$ lies in $W_{SO_{2n}}$ if and only if, $\prod_{i=1}^na_i = 1$ (see \cite{Car,LFV}). We identify $W_{SO_{2n}}$ with this index $2$ subgroup of $C_2\wr S_n$. 
Consequently we get that a conjugacy class of $C_2\wr S_n$ lies in $W_{SO_{2n}}$ if and only if its associated signed partition has even number of negative terms counting with multiplicity. We will denote the collection of signed partitions of $n$ with an even number of negative terms counting with multiplicity by  $\mathcal{P}_{n}^{\pm,0}$. 

We will need the following lemma.
\begin{lemma}\label{index2}
    Let $\mathbf{a}\in W_{SO_{2n}}$, then either $Z_{W_{SO_{2n}}}(\mathbf{a})=Z_{C_2\wr S_n}(\mathbf{a})$ or $Z_{W_{SO_{2n}}}(\mathbf{a})$ is an index $2$ subgroup of $Z_{C_2\wr S_n}(\mathbf{a})$.
\end{lemma}
\begin{proof}
    Since $W_{SO_{2n}}$ is a index $2$ subgroup of $C_2\wr S_n$ (and thus normal), from the second isomorphism theorem of groups we have
    \[
    Z_{C_2\wr S_n}(\mathbf{a})/(Z_{C_2\wr S_n}(\mathbf{a})\cap W_{SO_{2n}})
    \cong
    Z_{C_2\wr S_n}(\mathbf{a})W_{SO_{2n}}/W_{SO_{2n}}.
    \]
    Since $(Z_{C_2\wr S_n}(\mathbf{a})\cap W_{SO_{2n}})=Z_{W_{SO_{2n}}}(\mathbf{a})$ and $Z_{C_2\wr S_n}(\mathbf{a})W_{SO_{2n}}$ is either $C_2\wr S_n$ or $W_{SO_{2n}}$, the lemma follows.
\end{proof}
\begin{example}
\begin{enumerate}[(i)]
    \item If $\mathbf{a}\in W_{SO_{2n}}$ corresponds to the signed partition $k^{l}$, $k$ even, then $Z_{W_{SO_{2n}}}(\mathbf{a})=Z_{C_2\wr S_n}(\mathbf{a})$.
    \item When $\mathbf{a}\in W_{SO_{2n}}$ corresponds to the signed partition $k^{s}\overline{k}^{t}$, where either $k$ is odd, or $k$ is even and $t>0$, then $Z_{W_{SO_{2n}}}(\mathbf{a})$ is an index $2$ subgroup of $Z_{C_2\wr S_n}(\mathbf{a})$.
\end{enumerate}
\end{example}
\begin{theorem}
Let $\Gamma$ be a finitely generated nilpotent group of abelian rank $r\geq 1$.
The mixed Hodge polynomial $\mu_{\mathcal{M}^{0}_{\Gamma}SO_{2n}}(t,u,v)$ of $\mathcal{M}^{0}_{\Gamma}SO_{2n}$ has an explicit form as:
    \[
\sum_{\substack{\lambda^{\pm}\in\mathcal{P}^{\pm,0}_n:\\i\text{ even }:b_{i}=0\\i\text{ odd }:a_{i}=b_{i}=0}}\frac{\prod_{i=1, i\text{ even}}^{n}(1-(tuv)^{i})^{a_{i}r}}{\prod_{i=1,i\text{ even}}^{n}(2i)^{a_{i}}a_{i}!}
+
\sum_{\substack{\lambda^{\pm}\in\mathcal{P}^{\pm,0}_n:\\\exists i\text{ even }:b_{i\neq 0}\text{ or}\\\exists i\text{ odd }:(a_{i}b_{i})\neq(0,0)}}2\quad\frac{\prod_{i=1}^{n}(1-(-tuv)^{i})^{a_{i}r}\prod_{j=1}^{n}(1+(-tuv)^{j})^{b_{j}r}}{\prod_{i=1}^{n}(2i)^{a_{i}}a_{i}!\prod_{j=1}^{n}(2j)^{b_{j}}b_{j}!}
    .
    \]
\end{theorem}
\begin{proof}
As in the proof of Theorem \ref{sp2n}, the formula in Fact \ref{generalform} for $G=SO_{2n}$ can be written as
\[
    \mu_{\mathcal{M}^{0}_{\Gamma}SO_{2n}}(t,u,v)
    =\mu_{\mathcal{M}^{0}_{\mathbb{Z}^r}SO_{2n}}(t,u,v)
    =
    \sum_{\lambda^{\pm}\in\mathcal{P}^{\pm,0}_n}\frac{\big(\det(I+tuvA_{w_{\lambda^{\pm}}})\big)^{r}}{|Z_{W_{SO_{2n}}}(w_{\lambda^{\pm}})|},
    \]
where for each $\lambda^{\pm}=[1^{a_{1}}\overline{1}^{b_{1}}2^{a_{2}}\overline{2}^{b_{2}}\ldots n^{a_{n}}\overline{n}^{b_{n}}]\in\mathcal{P}^{\pm,0}_n$ we choose a representative $w_{\lambda^{\pm}}\in W_{SO_{2n}}$ of the corresponding conjugacy class in $W_{SO_{2n}}$ corresponding to $\lambda^{\pm}$ defined in the same way as in proof of Theorem \ref{sp2n}. We obtain
\[
\det(I+tuvA_{w_{\lambda^{\pm}}})=\prod_{i=1}^{n}(1-(-tuv)^{i})^{a_{i}}\prod_{j=1}^{n}(1+(-tuv)^{j})^{b_{j}}.
\]
By Fact \ref{centralizer} and Lemma \ref{index2}, we have $Z_{W_{SO_{2n}}}(w_{\lambda^{\pm}})=Z_{C_2\wr S_n}(w_{\lambda^{\pm}})$ if and only if $a_{k}=b_{k}=0$ when $k$ is odd and $b_{k}=0$ when $k$ is even, $1\leq k\leq n$. If this is the case, then $|Z_{W_{SO_{2n}}}(w_{\lambda^{\pm}})|=\prod_{i=1,i\text{ even}}^{n}(2i)^{a_{i}}a_{i}!$, otherwise,  $|Z_{W_{SO_{2n}}}(w_{\lambda^{\pm}})|=\frac{\prod_{i=1}^{n}(2i)^{a_{i}}a_{i}!\prod_{j=1}^{n}(2j)^{b_{j}}b_{j}!}{2}$. This gives the desired formula.
\end{proof}

\Addresses
\end{document}